\documentclass[12pt]{iopart}
\usepackage{amsfonts}
\usepackage{amssymb,amsthm}
\usepackage{cite}
\newtheorem{theorem}{Theorem}[section]       
\newtheorem{lemma}[theorem]{Lemma}           
\newtheorem{proposition}[theorem]{Proposition} 

\newcommand{\curl}{\mbox{curl}}

\begin{document}
\title{An integrating factor matrix method to find first integrals}

\author{K V I Saputra$^1$, G R W Quispel$^2$, L van Veen$^3$}

\address{$^1$Faculty of Science and Mathematics, University Pelita Harapan,
Jl. M.H. Thamrin Boulevard, Tangerang Banten, 15811 Indonesia}
\address{$^2$Department of Mathematics and Statistical Science,
La Trobe University, Bundoora 3086, Australia}
\address{$^3$Faculty of Science, University of Ontario Institute of Technology,
2000 Simcoe St. N., Oshawa, ON, Canada L1H 7K4}
\ead{kie.saputra@staff.uph.edu}
\begin{abstract}
In this paper we developed an integrating factor matrix method to derive conditions for the existence of first integrals. 
We use this novel method to obtain first integrals, along with the conditions for their existence, for two and three dimensional Lotka-Volterra systems with constant terms. The results are compared to previous results obtained by other methods.
\end{abstract}

\section{Introduction}
Consider a general dynamical system in $n$ dimensions as follows
\begin{equation}\label{gen}
    \dot{x} = f(x,\mu), \quad x \in \mathbb{R}^n \quad\mbox{with}\quad n >1,
\end{equation}
where $\mu \in \mathbb{R}^p$ is a vector of parameters. Our main objective is to find conditions on the parameters such that 
the system above possesses a first integral. In this paper, we concentrate on finding integrals of Lotka-Volterra systems with constant terms. 

The Lotka-Volterra system (LV) has been the subject of intensive studies during the past century. The interaction 
of two species in an ecosystem \cite{hofbauer}, a metamorphosis of turbulence in plasma physics \cite{ball}, 
hydrodynamic equations \cite{busse}, autocatalytic chemical reactions \cite{hering_roger} 
and many more, are of Lotka-Volterra type. 
Nevertheless, the dynamics of such systems is far from being understood. 
Finding first integrals of LV systems, or any dynamical system, gives global information about the long-term behaviour of such systems. 

In two-dimensional systems, the existence of a first integral implies that the system is completely integrable because the phase portraits are completely characterized. For three-dimensional systems, the existence of a first integral means that there cannot be chaotic motion as the solutions will live inside the level sets of such an integral function. Here we include constant terms to generalize LV systems. The 
constant term can be considered as a constant rate harvesting. Dynamics and bifurcation analysis of such a system have been studied in \cite{saputra}.

Many different methods have been developed to study the existence of first integrals of LV systems. 
Perhaps one of the earliest attempts to study the existence of first integrals was published by Cair{\'o} et al.~\cite{cairo89}, who studied the integrability of $n$-dimensional Lotka-Volterra equations using the Carleman embedding method. They sought an invariant that may 
be time-dependent. Cair{\'o} and Llibre~\cite{cairo00} used a polynomial inverse integrating factor to find a condition for the existence of the first integral. The Darboux method that uses the relationship between algebraic curves and integrability of differential equations has been introduced by Cair{\'o} and Llibre~\cite{cairo97} to study two-dimensional LV systems. Cair{\'o} et al. \cite{cairo99} used the same method to search 
for a first integral of two-dimensional quadratic systems. 

The Darboux method has also been used to derive an integral for three-dimensional LV systems~\cite{cairo_nonlinearmathphy00} and for the so-called ABC systems, which correspond to particular cases of three-dimensional LV. The ABC systems were among the first three-dimensional models that were investigated. Grammaticos et al.~\cite{grammaticos} derived first integrals using the Frobenius Integrabilty Theorem method (first introduced by Strelcyn and Wojciechowski~\cite{strelcyn88}). Ollagnier~\cite{ollagnier} has found polynomial first integrals of the ABC system. 
Gao and Liu~\cite{gao98} presented a method that basically relies on changing variables to transform three-dimensional LV systems to two-dimensional ones. The existence of first integrals follows from integrating the two-dimensional systems. Gao \cite{gao99} used a direct integration method to find first integrals of three-dimensional Lotka-Volterra systems. A new algorithm presented by Gonzalez-Gascon and Peralta Salas \cite{Gonzalez-Gascon00} also used three-dimensional LV systems as a test case.

The approach closest to the work presented here uses the idea of associating a Hamiltonian to a first integral of a vector field. It was introduced by Nutku \cite{nutku90}. A generalization of this idea to two-dimensional vector fields having a first integral was provided by Cair{\'o} and Feix~\cite{cairo92ontheHam}.
They showed that, through time rescaling, the first integral can be considered as a Hamiltonian. Subsequently, Cairo et al.~\cite{cairo93} and Hua et al.~\cite{hua93} used an Ansatz for their Hamiltonian functions. They assumed that a first integral (or an invariant) $H$ is a product of functions, $H = P(x,y)(Q(x,y))^{\mu}(R(x,y))^{\nu}$, where $P,Q,R$ are first degree polynomials, and derived conditions for two-dimensional quadratic systems to have a first integral.

Another Hamiltonian method that has been used works as follows. A general system (\ref{gen}) is said to have a Hamiltonian structure if and only if it can be written as $\dot{x}=f(x)=S(x)\nabla H(x)$, where $S$ is a skew-symmetric matrix and $H$ is a smooth function. The matrix function $S$ must satisfy the Jacobi identity~\cite{gao00}. Plank~\cite{plank} has used this property to find a Hamiltonian function for two-dimensional LV systems, while Gao~\cite{gao00}, using the same property, has derived conditions for three-dimensional systems not to be chaotic. 

In this paper, we will not impose the Jacobi identity on the matrix $S$ since we only want $H$ to be a first integral. Thus, it is sufficient to ensure that $f(x)$ can be written as $S(x)\nabla H(x)$. Consider the following proposition.
\begin{proposition}[McLachlan et al. 1999~\cite{mclachlan99}]
Let $f \in C^r(\mathbb{R}^n,\mathbb{R}^n)$, $r \geq 1$, $n>1$, be a vector field and $H \in C(\mathbb{R}^n,\mathbb{R})$ be a first integral of the vector field $f$ (i.e. $f \cdot \nabla H=0$) for all $x$. Then there is a skew-symmetric matrix function $S(x)$ on the domain $\{x:\nabla H \neq0\}$ such that $f=S\nabla H$.
\end{proposition}
As a consequence, there is also a skew symmetric matrix function $T(x)$ on the domain $\{x:f \neq0\}$ such that $\nabla H=T f$. We are going to use this idea to find first integrals and the associated constraints on the parameters for two- and three-dimensional LV systems with constant terms. We call the matrix $T$ an integrating factor matrix if and only if 
\begin{equation}
\curl(\,Tf\,)=0.
\end{equation}
Making an Ansatz concerning the integrating factor matrix $T$, 
we obtain both an integral and the conditions on the parameters for its existence. The meaning of curl depends on the dimension of the system. 
For two dimensional systems, $\curl(Tf)$ is the scalar 
\[
    \frac{\partial (Tf)_1}{\partial x_2} - \frac{\partial (Tf)_2}{\partial x_1} = 0,
\]
where $(TF)_i$ is the $i-th$ component of the vector $Tf$. In three-dimensional systems, $\curl(T f)=\nabla \times Tf$ as usual.

In Secs. \ref{2d} and \ref{3d} we will discuss the results of applying the integrating factor matrix approach to two- and three-dimensional
Lotka-Volterra systems with constant terms, respectively. 
In either case , we will impose some conditions on the integrating factor matrix. We will assume, that its entries
can be written as a product of functions of a single variable. In the three-dimensional case, we use two different
forms of the matrix. These assumptions are restrictive, and we do not claim to present the most general class
of first integrals which can be found with the proposed method. However, the most general form of the integrating
factor matrix will lead to conditions in the form of coupled, nonlinear partial differential equations. The analysis
of these equations is work in progress. We emphasize, that with the simplifying conditions used in the current paper,
we reproduce many known first integrals and uncover several new ones both in two and three dimensions
Each section is concluded with a comparison to earlier results from the literature.

\section{Two-dimensional LV systems with constant terms \label{2d}}
In this section, we consider integrals of the two-dimensional Lotka-Volterra system with constant terms:
\begin{eqnarray}\label{eq:lv_ef}
    &\dot{x_1} = f_1(x_1,x_2) = x_1(b_1 + a_{11}x_1 + a_{12}x_2) + e_1, \nonumber\\
    &\dot{x_2} = f_2(x_1,x_2) = x_2(b_2 + a_{21}x_1 + a_{22}x_2) + e_2,
\end{eqnarray}
where $b_i$, $a_{ij}$ ($i,j=1,2$) are arbitrary parameters and $e_1,e_2$ are the constant terms. 
We choose an integrating factor matrix as follows,
\begin{equation}\label{eq:intmatrix}
    T(x_1,x_2) = \left( \begin{array}{cc}
    0 & -R \\
      R & 0 \\
    \end{array} \right), 
\end{equation}
where the function $R=R(x_1,x_2)$ is to be determined later on. The matrix $T(x_1,x_2)$ is an integrating factor if and only if the curl of $T \emph{f}\,$ is zero, where $\emph{f}=(f_1,f_2)$. As is mentioned before, in the two dimensional case, this condition is equivalent to the following, 
\begin{equation}\label{eq:intfac}
\frac{\partial(Rf_1)}{\partial x_1}+\frac{\partial(Rf_2)}{\partial x_2}=0.
\end{equation}
The associated first integral $H$ is given by
\begin{equation}\label{eq:integral}
H(x_1,x_2)= \int R(x_1,x_2) f_1(x_1,x_2) \, d x_2 + h(x_1),
\end{equation} 
where $h(x_1)$ is found by imposing $\partial H/\partial x_1 = -Rf_2$.

Let us assume that $R$ is separable, i.e. $R(x_1,x_2)=A(x_1)B(x_2)$. We have the following lemma.
\begin{lemma}
	Let $f_1$ and $f_2$ be the functions given in the vector field (\ref{eq:lv_ef}), then the condition
	(\ref{eq:intfac}) determines the forms of $A(x_1)$ and $B(x_2)$ to be,
\begin{equation}\label{lemma1}
	A(x_1)	= \exp(\frac{\alpha x_1}{a_{12}}) {x_1}^{\beta/a_{12}}, \quad
	B(x_2)	= \exp(\frac{-\alpha x_2}{a_{21}}) {x_2}^{\gamma/a_{21}},
\end{equation}
where $\alpha$, $\beta$ and $\gamma$ are constants to be determined later on.
\end{lemma}
\begin{proof}
We substitute $R$ into 
(\ref{eq:intfac}) to obtain
\begin{equation}\label{eq:genA3}
	\frac{1}{A} \frac{d A}{d x_1}a_{12}x_1x_2 + \frac{1}{B} \frac{d B}{d x_2} a_{21}x_1x_2 + F(x_1) + G(x_2) = 0,
\end{equation}
where 
\begin{eqnarray}\label{eq:FG}
	F(x_1) &= &\frac{A_1}{A}(b_1x_1+a_{11}{x_1}^2+e_1)+b_1+2a_{11}x_1+a_{21}x_1, \nonumber\\
	G(x_2) &= &\frac{B_2}{B}(b_2x_2+a_{22}{x_2}^2+e_2)+b_2+2a_{22}x_2+a_{12}x_2. 
\end{eqnarray}
Applying $\partial^2/(\partial x_1 \partial x_2)$ to both sides of 
equation (\ref{eq:genA3}), we obtain a separable differential equation in terms of 
$A(x_1)$ and $B(x_2)$, which can be solved explicitly. 
\end{proof}
We then substitute $A(x_1)$ and $B(x_2)$ to the equation (\ref{eq:genA3}) to obtain
\begin{equation}
	\gamma x_1 + F(x_1) + \beta x_2 + G(x_2) = 0,
\end{equation}
or, 
\begin{equation}
	\gamma x_1 + F(x_1) = \zeta, \quad \beta x_2 + G(x_2) = -\zeta,
\end{equation}
where $\zeta$ is a constant. Finally, we obtain the following set of equations:
\begin{equation}
\hspace{-2cm}(\frac{\alpha e_1}{a_{12}} + \frac{\beta b_1}{a_{12}} + b_1) + (\gamma +\frac{\alpha b_1}{a_{12}} +\frac{\beta a_{11}}{a_{12}} + 2a_{11}+a_{21})x_1 
+  \frac{\alpha a_{11}}{a_{12}}{x_1}^2 +   \frac{\beta e_1}{a_{12}x_1}  = \zeta, 
\end{equation}
and respectively, 
\begin{equation}
\hspace{-2cm}(- \frac{\alpha e_2}{a_{21}} + \frac{\gamma b_2}{a_{21}} + b_2) + (\beta  - \frac{\alpha b_2}{a_{21}} + \frac{\gamma a_{22}}{a_{21}} +2a_{22}+a_{12})x_2 
- \frac{\alpha a_{22}}{a_{21}}{x_2}^2 + \frac{\gamma e_2}{a_{21}x_2}  = -\zeta.
\end{equation}
The expression above is satisfied if and only if the parameters satisfy the following conditions:
\begin{eqnarray}			
	&0 = \frac{\alpha a_{11}}{a_{12}} = \frac{\alpha a_{22}}{a_{21}} \label{two}\\	
	&0 = \frac{\beta e_1}{a_{12}} = \frac{\gamma e_2}{a_{21}} \label{three}\\	
	&0 = \gamma +\frac{\alpha b_1}{a_{12}} +\frac{\beta a_{11}}{a_{12}} + 2a_{11}+a_{21} \label{one}\\
	&0 = \beta  - \frac{\alpha b_2}{a_{21}} + \frac{\gamma a_{22}}{a_{21}} +2a_{22}+a_{12} \label{four}\\
	&\frac{\alpha e_1}{a_{12}} + \frac{\beta b_1}{a_{12}} + b_1 = \frac{\alpha e_2}{a_{21}} - \frac{\gamma b_2}{a_{21}} - b_2. \label{seven}
\end{eqnarray}
Let us introduce $l_1:=\beta/a_{12}+1$ and $l_2:=\gamma/a_{21}+1$. When $\alpha=0$, 
the system (\ref{three}-\ref{seven}) is equivalent to the following system of equation:
\begin{eqnarray}
    e_1 l_1 = e_1, \label{cond1} \\ 
    e_2 l_2 = e_2, \label{cond2} \\
    l_1a_{11}+l_2a_{21}=-a_{11}, \label{cond3}\\
    l_1a_{12}+l_2a_{22}=-a_{22}, \label{cond4}\\
    b_1l_1 + b_2l_2=0,\label{cond5}
\end{eqnarray}
which is more comfortable to work on and closely related to Plank \cite{plank}. 
As a consequence, in the case of $\alpha=0$, our problem can be divided 
into three different subcases, depending on the value of the constant terms. In the following we give the corresponding conditions along with the resulting integrals for the cases  $e_1,e_2\neq 0$, $e_1\neq 0,e_2=0$ and $e_1=e_2=0$ separately. The case $e_1=0, e_2 \neq 0$ follows by symmetry considerations. We note that the case corresponding to the original Lotka-Volterra system where $e_1=e_2=0$, has been discussed by various people. 
However we have included a discussion of this case including some previously unknown special solutions. The case 
where $\alpha \neq 0$ is discussed in the subsection (2.4). 

To start with, let us consider (\ref{cond1}-\ref{cond5}) as an overdetermined linear system:
\begin{equation}\label{lin_sys}
   A l = r,
\end{equation}
where $l=(l_1,l_2)$ and the matrix $A$ and the vector $r$ are to be determined later. The system has a solution only if the vector $\emph{r}$ is orthogonal to the left null space of the matrix $A$. 

\subsection{The case $e_1,e_2\neq0$}\label{case3}
We have the case where both the constant terms $e_1$ and $e_2$ are non-zero. Then by (\ref{cond1}) and (\ref{cond2}), this implies that $l_1=l_2=1$.
Moreover, we can simplify the other conditions to:
\begin{equation}\label{cond:notzero_ef}
    b_1+b_2=0, \quad 2a_{11}+a_{21}=0, \quad \textrm{and} \quad  a_{12}+2a_{22}=0.
\end{equation}
If the Lotka-Volterra system (\ref{eq:lv_ef}) with non-zero constant terms $e$ and $f$ satisfies conditions (\ref{cond:notzero_ef}), then it has a first integral that is given by:
\begin{equation}\label{int_ef}
    H = b_1x_1x_2 + a_{11}x_1^{\phantom{1}2}x_2 - a_{22}x_1x_2^{\phantom{1}2} + e_1x_2 - e_2x_1. 
\end{equation}

\subsection{The case $e_1\neq0$, $e_2=0$}\label{case2}
One of the constant terms, $e_1$ is not zero. This means that the free parameter $l_1$ must be 1 by (\ref{cond1}). We now have a linear system like (\ref{lin_sys}) with a 3 by 1 matrix $A$ and a vector $\emph{r}$ in $\mathbb{R}^3$ with only one unknown $l_2$. Without loss of generality, we assume that the matrix A is of rank 1, when A has rank zero we have a trivial integral $H=x_2$ since $\dot{x}_2=0$. Using the fact that this system must be solvable, we can again find the conditions for the existence of the first integral. As we assume that $a_{21}\neq 0$, the solvability conditions are given by:
\begin{equation}\label{cond:notzero_e}
    \frac{2a_{11}a_{22}}{a_{21}}-a_{22}-a_{12}=0, \quad \textrm{and} \quad
    2b_2a_{11}-b_1a_{21}=0,
\end{equation}
and we have $l_2=-2a_{11}/a_{21}$.
If $l_2$ is not zero then the first integral is given by:
\begin{equation}\label{int_e1}
    H = x_2^{\phantom{1}l_2}(-b_2x_1 - \frac{a_{21}}{2}x_1^{\phantom{1}2} 
    - a_{22}x_1x_2 + \frac{e_1}{l_2}).
\end{equation}
However, in the case where the exponent $l_2$ is zero, the integral is given by:
\begin{equation}\label{int_e2}
    H = -b_2x_1 - \frac{a_{21}}{2}x_1^{\phantom{1}2} - a_{22}x_1x_2 + e_1\ln|x_2|. 
\end{equation}
Note, that expression (\ref{int_e1}) is also obtained as a limit of result (\ref{int_ef}) but the
conditions stated in the case $e_1\neq 0$ and $e_2=0$ are more general than those obtained
in the limit.

\subsection{The case $e_1=0$ and $e_2=0$}\label{case1}
Finally, if $e_1=e_2=0$ equations (\ref{cond1}) and (\ref{cond2}) are trivial and we have a linear system of the form (\ref{lin_sys}) with a 3 by 2 matrix $A$ and a vector $\emph{e}$ in $\mathbb{R}^3$ from (\ref{cond3}-\ref{cond5}).

If A is of maximal rank then the solvability condition of the linear system (\ref{lin_sys}) is given by :
\begin{equation}\label{cond:zero_ef}
    b_1a_{22}(a_{21}-a_{11})+b_2a_{11}(a_{12}-a_{22})=0.
\end{equation}
If our system satisfies this condition  then
\begin{equation}
l_1=\frac{a_{22}(a_{21}-a_{11})}{a_{11}a_{22}-a_{12}a_{21}} \qquad \mbox{and}\qquad l_2=\frac{a_{11}(a_{12}-a_{22})}{a_{11}a_{22}-a_{12}a_{21}}.
\end{equation}
When neither $l_1$ nor $l_2$ are zero, the first integral is given by:
\begin{equation}\label{int1}
    H = x_1^{\phantom{1}l_1}x_2^{\phantom{1}l_2}(\frac{b_1}{l_2} + \frac{a_{11}}{l_2}x_1 
    - \frac{a_{22}}{l_1}x_2). 
\end{equation}
However, either $l_1$ or $l_2$ may be zero and if $l_1 =0$, $l_2\neq0$ and $l_2\neq-1$ then we have $b_2=a_{22}=0$, $a_{11} \neq a_{21}$ and an integral that is given by:
\begin{equation}\label{int2}
    H = x_2^{\phantom{1}l_2}(\frac{b_1}{l_2} + \frac{a_{11}}{l_2}x_1 
    + \frac{a_{12}}{(l_2+1)}x_2). 
\end{equation}
But when $l_1 =0$ and $l_2=-1$, this implies $b_2=0$ and $a_{21}=a_{11}$. It follows that the integral is given by:
\begin{equation}\label{int3}
    H = a_{12}\ln |x_2| - a_{22}\ln|x_1| - \frac{b_1}{x_2} - a_{11}\frac{x_1}{x_2}.
\end{equation}
Finally, when both $l_1$ and $l_2$ are zero, which implies $a_{11}=a_{22}=0$, the first integral is given by:
\begin{equation}\label{int4}
    H = b_1\ln |x_2| + a_{12} x_2 - b_2\ln |x_1| - a_{21} x_1.
\end{equation}
We remark that the case when $l_2 =0$, $l_1\neq0$, $l_1\neq-1$ and the case when $l_2 =0$, $l_1=-1$ follow by symmetry considerations. 

Now let us assume that $A$ has rank one. Excluding the trivial case in which one column vector
is zero, we assume that
\begin{equation}
    (a_{11},a_{12},b_{1})^T = \lambda (a_{21}, a_{22},b_{2})^T,
\end{equation}
for some $\lambda \neq 0$. This leads to the following integral:
\begin{equation}
	H=x_1x_2^{-\lambda}, \quad \textrm{where } \lambda=a_{11}/a_{21}.
\end{equation} 

\subsection{The case where $\alpha \neq 0$} 
This implies $a_{11}=a_{22}=0$, due to conditions (\ref{two}). 
Consider conditions (\ref{one}) and (\ref{four}). Eliminating $\alpha$ in both equations gives:
\begin{equation}\label{onefour}
(\beta+a_{12})b_1a_{21}+(\gamma+a_{21})b_2a_{12}=0.
\end{equation}
Consider now condition (\ref{seven}). After some algebraic manipulation we get:
\begin{equation}\label{seven2}
(\beta+a_{12})b_1a_{21}+(\gamma+a_{21})b_2a_{12}=\alpha(e_2a_{12}-e_1a_{21}),
\end{equation}
and due to the fact that $\alpha \neq 0$, we have
\begin{equation}\label{seven3}
e_2a_{12}-e_1a_{21} = 0.
\end{equation}
Finally conditions (\ref{three}) gives the following three sub-cases:
\begin{equation}
\textrm{(a) } \gamma=\beta=0, \quad \textrm{(b) } \gamma \neq 0, \,\, \beta=0, \quad 
\textrm{and (c) } \gamma, \beta \neq 0
\end{equation}
The case where $\beta \neq 0$ and $\gamma=0$ follows by symmetry. 
Note that the conditions in subcases (b) and (c) imply that $e_1=e_2=0$. These subcases are already discussed and the integrals are equivalent to the integral in equation (\ref{int4}). The only case remaining is when $\gamma=\beta=0$ that gives $b_1+b_2=0$, provided $a_{12}a_{21}\neq 0$. We then obtain the following first integral:
\begin{equation*}
H = e^{(a_{21}x_1 - a_{12}x_2)/b_2}  (e_1+a_{12}x_1x_2),
\end{equation*}
or, equivalently,
\begin{equation}
H = a_{21}x_1 - a_{12}x_2 +b_2 \ln (e_1+a_{12}x_1x_2).
\end{equation}

\subsection{Further notes regarding the known integrals of two-dimensional LV systems and quadratic systems}
Many attempts have been made to study the integrability of two-dimensional LV systems and general quadratic systems. Different first integrals were found using different methods, some similar to those found in this discussion.

The first integral (\ref{int_ef}) along with its conditions (\ref{cond:notzero_ef}) was
probably first found by Frommer in 1934 (see Art{\'e}s and Llibre \cite{artes94}). It was also derived by Cair{\'o} {\sl et al.}, who used the Hamiltonian method and by Hua {\sl et al.}, who studied the connection between the existence of a first integral and the Painlev{\'e} property in a general quadratic system. In the latter work, the form of the first integral and the vector field are different, but through some invertible transformations it is not hard to check that the result is actually equivalent.

Many first integrals were known for 
the case $e_1=0$ and $e_2=0$. For instance, (\ref{int1}-\ref{int4}) were derived by Nutku \cite{nutku90}, Plank \cite{plank}
and Cair{\'o} {\sl et al.} \cite{cairo92families,cairo99,cairo00} using various methods. We remark that the first integral (\ref{int4}) that has constraints $a_{11}=a_{22}=0$ was first derived by Volterra himself as a constant of motion (see the book by Hofbauer and Sigmund \cite{hofbauer}). 
To our best knowledge, the other first integrals presented here are new.

\section{Three-dimensional LV systems with constant terms \label{3d}}
We consider the following three-dimensional LV systems with constant terms:
\begin{eqnarray}\label{eq:lv_e3}
    &\dot{x_1} = f_1(x_1,x_2,x_3) = x_1(b_1 + a_{11}x_1 + a_{12}x_2 + a_{13}x_3) + e_1, \nonumber\\
    &\dot{x_2} = f_2(x_1,x_2,x_3) = x_2(b_2 + a_{21}x_1 + a_{22}x_2 + a_{23}x_3) + e_2, \nonumber\\
    &\dot{x_3} = f_3(x_1,x_2,x_3) = x_3(b_3 + a_{31}x_1 + a_{32}x_2 + a_{33}x_3) + e_3,
\end{eqnarray}
where $b_i$, $a_{ij}$ ($i,j=1,2,3$) are arbitrary parameters and $e_i$ ($i=1,2,3$) are the constant terms. 
In this section, in order to find integrals of the system above we shall make the following two Ansatzs for the skew-symmetric matrix $T$, 
\begin{equation}\label{T1}
    T_1(x_1,x_2,x_3) = R
    \left( \begin{array}{ccc}
    0 & -\alpha' & -\beta' \\
    \alpha' & 0 & -\gamma' \\
    \beta'  & \gamma' & 0 \\
    \end{array} \right),
\end{equation}
respectively, 
\begin{equation}\label{T2}
    T_2(x_1,x_2,x_3) = R
    \left( \begin{array}{ccc}
    0 & -\alpha x_3 & -\beta x_2 \\
    \alpha x_3 & 0 & -\gamma x_1 \\
    \beta x_2  & \gamma x_1 & 0 \\
    \end{array} \right),
\end{equation}
where $\alpha,\alpha',\beta,\beta',\gamma,\gamma' \in \mathbb{R}$ are arbitrary parameters. 
We also use an ansatz for the function $R$, namely 
$R=R(x_1,x_2,x_3)=x_1^{l_1-1}x_2^{l_2-1}x_3^{l_3-1}$, where the $l_i$ ($i=1,2,3$) are free 
parameters that are to be determined later on. As discussed in the introduction, this Ansatz is restrictive.
However, a more general treatise falls outside the scope of the present paper. 

The matrices $T_i(x_1,x_2,x_3)$ ($i=1,2$) are 
integrating factors if and only if $\curl(T_i \emph{f})=0 \,$, where 
$\emph{f}=(f_1,f_2,f_3)^T$. 
In the three-dimensional case, this condition is equivalent to 
\begin{equation}\label{eq:intfac3}
 \left| \begin{array}{ccc}
 i & j & k \\
 \partial/\partial x_1 & \partial/\partial x_2 & \partial/\partial x_3 \\
 \partial H_i/\partial x_1 & \partial H_i/\partial x_2 & \partial H_i/\partial x_3 \\
 \end{array} \right| = 0,
\end{equation}
where 
\begin{equation}
    \nabla H_i = T_i f \quad(i=1,2). 
\end{equation}
In order to find the first integral $H_1$,
we expand the above expression with respect to the matrix $T_1$ as follows,
\begin{eqnarray} 
\frac{\partial H_1}{\partial x_1} = -R \alpha' f_2 - R \beta' f_3, \label{eq:H1i}\\
\frac{\partial H_1}{\partial x_2} =  R \alpha' f_1 - R \gamma' f_3, \label{eq:H1j}\\
\frac{\partial H_1}{\partial x_3} =  R \beta' f_1 + R \gamma' f_2. \label{eq:H1k}
\end{eqnarray}
and the associated first integral $H_1(x_1,x_2,x_3)$ is given by
\begin{equation}\label{eq:integral3}
H_1(x_1,x_2,x_3)= \int (R \beta' f_1 + R \gamma' f_2) \, d x_3 + h(x_1,x_2),
\end{equation} 
where $h(x_1,x_2)$ is found by imposing (\ref{eq:H1i}) and (\ref{eq:H1j}). The computation
of $H_2$ is completely analogous.

In the following, we shall derive integrals for the cases $e_1, e_2, e_3 \neq 0$; $e_1, e_2 \neq 0, e_3=0$; 
$e_1 \neq 0, e_2=e_3=0$;  and $e_1=e_2=e_3=0$. 

\subsection{The case $e_1,e_2,e_3 \neq 0$}
A bit of algebra shows that if $e_1,e_2,e_3 \neq 0$, the integrating factor matrix $T_2$ does not yield
any solutions, so, in this case, we will only use $T_1$.
We substitute expressions (\ref{eq:H1i}-\ref{eq:H1k}) in the condition (\ref{eq:intfac3}) and obtain
\begin{equation}\label{cond:curl1}
   \left( \begin{array}{c}
   \vspace{1mm}
   \displaystyle{\frac{\partial}{\partial x_2}(R \beta' f_1 + R \gamma' f_2)-\frac{\partial}{\partial x_3}(R \alpha' f_1 - R \gamma' f_3)}\\
   \vspace{1mm}
   \displaystyle{\frac{\partial}{\partial x_3}(-R \alpha' f_2 - R \beta' f_3) - \frac{\partial}{\partial x_1}(R \beta' f_1 + R \gamma' f_2)}\\
   \displaystyle{\frac{\partial}{\partial x_1}(R \alpha' f_1 - R \gamma' f_3) - \frac{\partial}{\partial x_2}(-R \alpha' f_2 - R \beta' f_3)} \\
   \end{array} \right) = 0.
\end{equation}
Finally, we substitute the vector field (\ref{eq:lv_e3}) and the function $R$ into the above vector. For the first component we find 
\begin{eqnarray}\label{curl1a}
\frac{\beta'x_1}{x_2}[(l_2-1)b_1+(l_2-1)a_{11}x_1+l_2a_{12}x_2+(l_2-1)a_{13}x_3+(l_2-1)\frac{e_1}{x_1}]+ & \nonumber\\
\gamma'[l_2b_2+l_2a_{21}x_1+(l_2+1)a_{22}x_2+l_2a_{23}x_3+(l_2-1)\frac{e_2}{x_2}]- & \nonumber\\
\frac{\alpha'x_1}{x_3}[(l_3-1)b_1+(l_3-1)a_{11}x_1+(l_3-1)a_{12}x_2+l_3a_{13}x_3+(l_3-1)\frac{e_1}{x_1}]+ & \nonumber\\
\gamma'[l_3b_3+l_3a_{31}x_1+l_3a_{32}x_2+(l_3+1)a_{33}x_3+(l_3-1)\frac{e_3}{x_3}]=0 &.
\end{eqnarray}
and the other two components yield similar equations.
We now want to find conditions on the parameters ($\alpha',\beta',\gamma',l_i,a_{ij},b_i,e_i$) such that a solution exists.
The results for this case are summarized in the following lemma.
\begin{lemma}\label{lemma111}
The vector field (\ref{eq:lv_e3}) with $e_1,e_2,e_3\neq 0$ has a first integral in the following cases:
\begin{enumerate}
\item if the conditions $b_1+b_2=0$, $2a_{11}+a_{21}=0$, $2a_{22}+a_{12}=0$, $a_{13}=a_{23}=0$ are satisfied, then 
 the integral is given by
\begin{equation}\label{int111a}
    H=b_1x_1x_2+a_{11}x_1^2x_2-a_{22}x_1x_2^2+e_1x_2-e_2x_1,
\end{equation}
\item if the conditions $b_1+b_2=0$, $b_1+b_3=0$, $2a_{11}+a_{21}=0$, $2a_{11}+a_{31}=0$, $2a_{22}+a_{12}=0$, $2a_{33}+a_{13}=0$, and 
$a_{12}a_{13}+a_{12}a_{23}+a_{13}a_{32}=0$ are satisfied then the integral is given by
\begin{eqnarray}\label{int111b}
    H = &-2b_1a_{33}x_1x_3-2b_1a_{22}x_1x_2-2a_{11}a_{33}x_1^2x_3+2a_{11}a_{22}x_1^2x_2 \nonumber\\
    &+ 2a_{33}^2x_1x_3^2+2a_{22}^2x_1x_2^2+4a_{22}a_{33}x_1x_2x_3 \nonumber\\
    &+2(e_3a_{33}+e_2a_{22})x_1- 2e_1a_{33}x_3+2e_1a_{22}x_2,
\end{eqnarray}
\item if the conditions $b_i=0$ and $a_{ij} = -2a_{jj}$ for $i \neq j$ and $i,j=1,2,3$ are satisfied, then the integral is given by
\begin{eqnarray}\label{int111c}
    H &= a_{11}^2a_{22}x_1^2x_2-a_{11}^2a_{33}x_1^2x_3-a_{11}a_{22}^2x_1x_2^2+a_{11}a_{33}^2x_1x_3^2+ a_{22}^2a_{33}x_2^2x_3 \nonumber\\
    &- a_{22}a_{33}^2x_2x_3^2 + (-a_{11}a_{22}e_2+a_{11}a_{33}e_3)x_1 + (a_{11}a_{22}e_1-a_{22}a_{33}e_3)x_2  \nonumber\\
    &+ (a_{22}a_{33}e_2-a_{11}a_{33}e_1)x_3.
\end{eqnarray}
\end{enumerate}
All other first integrals that can be found using Ansatz (\ref{T1}) are related to these three cases through a permutation
of the coordinates $x_{i}$ and coefficients $\{a_{ij},b_{i},e_{i}\}$.
\end{lemma}
\begin{proof}
First considering terms proportional to $x_{i}^2/x_{j}$ for $i\neq j$ we find that $l_1=l_2=l_3=1$.
Consequently, we get the following conditions on the parameters,
\begin{eqnarray}
\alpha'(b_1+b_2)=0, \label{cond0a}\\
\beta'(b_1+b_3)=0, \label{cond0b}\\
\gamma'(b_2+b_3)=0, \label{cond0c}\\
\alpha'(2a_{11}+a_{21})=0, \label{cond0d}\\
\alpha'(2a_{22}+a_{12})=0, \label{cond0e}\\
\beta'(2a_{11}+a_{31})=0, \label{cond0f}\\
\beta'(2a_{33}+a_{13})=0, \label{cond0g}\\
\gamma'(2a_{22}+a_{32})=0, \label{cond0h}\\
\gamma'(2a_{33}+a_{23})=0, \label{cond0i}\\
-\alpha' a_{13}+\beta' a_{12}+\gamma'(a_{21}+a_{31})=0, \label{cond0j}\\
\alpha' a_{23}+\beta'(a_{12}+a_{32})+\gamma' a_{21}=0, \label{cond0k}\\
\alpha'(a_{13}+a_{23})+\beta' a_{32}-\gamma' a_{31}=0\label{cond0l}.
\end{eqnarray}
\begin{enumerate}
\item We start with the case where  $\alpha' \neq 0, \beta'=\gamma'=0$. From (\ref{cond0a}-\ref{cond0l}) the following conditions immediately apply:
\begin{eqnarray}
    &b_1+b_2=0, 2a_{11}+a_{21}=0, \label{cond12}\\
    &2a_{22}+a_{12}=0, a_{13}=a_{23}=0, \label{cond14}
\end{eqnarray}
and using (\ref{eq:integral3}) we obtain the first integral (\ref{int111a}). 
\item We turn to the case 
where $\alpha',\beta' \neq 0, \gamma' = 0$. From (\ref{cond0a}-\ref{cond0i}), the following conditions 
immediately apply,
\begin{eqnarray}
    &b_1+b_2=0, 2a_{11}+a_{21}=0, 2a_{11}+a_{31}=0, \\
    &b_1+b_3=0, 2a_{22}+a_{12}=0, 2a_{33}+a_{13}=0. 
\end{eqnarray}
While, $\alpha'$ and $\beta'$ can be computed from the following linear homogeneous equations due to (\ref{cond0j}-\ref{cond0l}), 
\begin{eqnarray}
    \alpha' a_{13} - \beta' a_{12} = 0, \label{lin1}\\
    \alpha' a_{23} + \beta' (a_{12}+a_{32}) = 0, \label{lin2} \\
    \alpha' (a_{13}+a_{23}) + \beta' a_{32} = 0. \label{lin3}
\end{eqnarray}
The solvability condition of the linear system above is given by, 
\begin{equation}\label{cond21}
    a_{12}a_{13}+a_{12}a_{23}+a_{13}a_{32}=0.
\end{equation}
If the above condition is satisfied, we obtain the first integral (\ref{int111b}) due to (\ref{eq:integral3}). 
\item Finally we consider the case $\alpha',\beta',\gamma' \neq 0$. From the equations (\ref{cond0a}-\ref{cond0i}), we have
\begin{eqnarray}
    b_1+b_2=0, 2a_{11}+a_{21}=0, 2a_{22}+a_{12}=0,\\
    b_1+b_3=0, 2a_{11}+a_{31}=0, 2a_{33}+a_{13}=0,\\
    b_2+b_3=0, 2a_{22}+a_{32}=0, 2a_{33}+a_{23}=0.
\end{eqnarray}
We simplify the equations above to
\begin{equation}\label{subsi1}
b_i=0 \quad \mbox{and} \quad a_{ij} = -2a_{jj}, \quad i \neq j,\,\, (i,j=1,2,3).
\end{equation}
The parameters $\alpha',\beta'$ and $\gamma'$ can be computed from the equations (\ref{cond0j}-\ref{cond0l}), 
giving us the following homogenous linear system,
\begin{equation}\label{linsys2}
    \left( \begin{array}{ccc}
      -a_{13} & a_{12} & (a_{21}+a_{31}) \\
      a_{23} & (a_{12}+a_{32}) & a_{21} \\
      (a_{13}+a_{23}) & a_{32} & -a_{31} \\
    \end{array} \right) 
    \left( \begin{array}{c}
      \alpha \\
      \beta \\
      \gamma \\
    \end{array} \right) = 0.
\end{equation}
Substituting (\ref{subsi1}) to the above linear system, we have the following solutions,
\begin{equation}
    \alpha=a_{11}a_{22}\mu,\quad \beta=-a_{11}a_{33}\mu,\quad \gamma=a_{22}a_{33}\mu.
\end{equation}
Then if all the parameters of the Lotka-Volterra systems (\ref{eq:lv_e3}) satisfy conditions above, the system (\ref{eq:lv_e3}) admits the first integral (\ref{int111c}).
\end{enumerate} 
We have ordered the solutions according to $\alpha'\neq 0$, $\alpha', \beta'\neq 0$ and $\alpha', \beta', \gamma'\neq 0$. A straightforward
computation shows that any other combination yields one of the  first integrals (\ref{int111a}-\ref{int111c})  with the coordinates and coefficients permuted. For instance, if we
take $\alpha'= \gamma'= 0$ and $\beta'\neq 0$, we find the first integral (\ref{int111a}) with the indices permuted according to $\{1,2,3\}\rightarrow \{1,3,2\}$.\qedhere
\end{proof} 

We remark that the first integral given in Lemma \ref{lemma111} point 1 is the same integral that is given in the two-dimensional case. This can be guessed as the first two equations of (\ref{eq:lv_e3}) are independent of $x_3$.

\vspace{0.5cm}

In the next three subsections, we now turn to the case where there is at least one constant term equal to zero. 
As it turns out, the first integrals found with Ansatz (\ref{T1}) can all be obtained as limits of the results
given in Lemma \ref{lemma111}. New first integrals can be found using Ansatz (\ref{T2}). Condition (\ref{eq:intfac3}) leads to
\begin{eqnarray}
B_1 l_1+B_2 l_2&=0, \label{cond6a}\\
B_3 l_1+B_2 l_3&=0, \label{cond6b}\\
B_3 l_2-B_1 l_3&=0, \label{cond6c}\\
A_{13} l_1+A_{23} l_2&=0, \label{cond6d}\\
A_{32} l_1+A_{22} l_3&=0, \label{cond6e}\\
A_{31} l_2-A_{11} l_3&=0, \label{cond6f}\\
A_{11} l_1+A_{21} l_2&=-A_{11}, \label{cond6g}\\
A_{12} l_1+A_{22} l_2&=-A_{22}, \label{cond6h}\\
A_{31} l_1+A_{21} l_3&=-A_{31}, \label{cond6i}\\
A_{33} l_1+A_{23} l_3&=-A_{23}, \label{cond6j}\\
A_{32} l_2-A_{12} l_3&=-A_{32}, \label{cond6k}\\
A_{33} l_2-A_{13} l_3&=A_{13}, \label{cond6l} \\
\gamma e_2(l_2-1)/x_2 &= \gamma e_3(l_3-1)/x_3 = (\beta l_2 - \alpha l_3)e_1/x_1 = 0,\label{cond6m}\\
\beta e_1(l_1-1)/x_1  &= \beta e_3(l_3-1)/x_3 = (\alpha l_3 + \gamma l_1)e_2/x_2 = 0,\label{cond6n}\\
\alpha e_1(l_1-1)/x_1 &= \alpha e_2(l_2-1)/x_2 = (\beta l_2 - \gamma l_1)e_3/x_3 = 0. \label{cond6o}
\end{eqnarray}
Here, we have introduced
\begin{eqnarray}\label{term_table}
B_1  &=b_1\alpha-b_3\gamma &\qquad  A_{1i}=a_{1i}\alpha-a_{3i}\gamma \nonumber\\
B_2  &=b_2\alpha+b_3\beta  &\qquad  A_{2i}=a_{2i}\alpha+a_{3i}\beta \nonumber\\
B_3  &=b_1\beta+b_2\gamma  &\qquad  A_{3i}=a_{1i}\beta+a_{2i}\gamma.
\end{eqnarray}
Note that 
\begin{eqnarray*}
B_1 \beta + B_2 \gamma-B_3 \alpha &=0 \\
A_{1i} \beta + A_{2i} \gamma-A_{3i} \alpha &=0 \quad \forall i.
\end{eqnarray*}

\subsection{The case $e_1,e_2\neq 0, e_3=0$}
Given that $e_1,e_2\neq 0, e_3=0$, we conclude that $l_1=l_2=1$ due to conditions (\ref{cond6m}-\ref{cond6o}), 
and $l_3$ must satisfy the following equations:
\begin{equation}
    \beta - \alpha l_3 = \alpha l_3 + \gamma = 0.
\end{equation}
The results are described in the following lemma. We omit the proof, which is straightforward and analogous to that
of Lemma \ref{lemma111}.
\begin{lemma}\label{lemma110}
The vector field (\ref{eq:lv_e3}) with $e_1,e_2\neq 0, e_3=0$ 
has the following first integrals:
\begin{enumerate}
\item $H=b_1x_1x_2+a_{11}x_1^2x_2-a_{22}x_1x_2^2+e_1x_2-e_2x_1$, under conditions \\
$b_1+b_2=2a_{11}+a_{21}=2a_{22}+a_{12}=0$ and $a_{13}=a_{23}=0$;
\item $H=(a_{13}-a_{23})x_1x_2x_3^2/2 + e_1 x_2x_3 - e_2 x_1x_3$, under conditions \\
$a_{13}-a_{23}\neq 0$, $b_1+b_3=b_2+b_3=0$, $a_{21}+a_{31}=a_{11}-a_{21}=0$,\\
$a_{22}+a_{32}=a_{12}-a_{32}=0$, and $a_{13}+a_{23}+2a_{33}=0$;
\item $H=(e_1 x_2 - e_2 x_1)x_3^{l_3}$,
where $l_3$ is a solution of 
\[
b_1-b_3l_3=a_{11}+a_{31}l_3=a_{12}+a_{32}l_3=a_{13}+a_{33}l_3=0.
\]
under conditions
$b_3^2+a_{31}^2+a_{32}^2+a_{33}^2\neq 0$, $b_1-b_2=0$, $a_{1i}-a_{2i}=0$\\
and $b_3a_{1i}-b_1a_{3i}=0\ \mbox{for}\ i=1,2,3$
\end{enumerate}
These are all first integrals that can be found using Ansatz (\ref{T2}). 
\end{lemma}
We also remark that the first integral given in Lemma \ref{lemma110} point 1 is trivial as it follows directly from the two-dimensional case.

\subsection{The case $e_1\neq 0, e_2=e_3=0$}
In this case, it immediately follows that $l_1=1$ due to conditions (\ref{cond6n}) and (\ref{cond6o}), 
and we have also the following equation to satisfy condition (\ref{cond6m}),
\begin{equation}
    \beta l_2 - \alpha l_3 = 0. 
\end{equation}
All results for this case are given in the following lemma,
\begin{lemma}\label{lemma100}
The vector field (\ref{eq:lv_e3}) with $e_1\neq 0, e_2=e_3=0$ 
has the following first integrals:
\begin{enumerate}
\item[1.] $H= -B_2x_1 - A_{21}x_1^2/2  + \alpha e_1 \ln |x_2| + \beta e_1 \ln |x_3|$, 
where $\alpha,\beta$ are solutions of $A_{22} = 0$, $A_{23}=0$, under conditions\\
$b_1=a_{11}=a_{12}=a_{13}=0$ and $a_{22}a_{33}-a_{23}a_{32}=0$;
\item[2.] $H= -a_{23}x_1x_3+e_1 \ln|x_2|$, under conditions\\
$b_2=a_{21}=a_{22}=0$, $a_{23}\neq0$ and $b_1+b_3=a_{1i}+a_{3i}=0 \ \mbox{for} \ i=1,2,3$;
\item[3.] $H=\beta \ln|x_3| + \alpha \ln|x_2|$, where $\alpha,\beta$ are solutions of $B_2=0$ and $A_{2i}=0$,\\
for $i=1,2,3$ under the condition $(b_2,a_{21},a_{22},a_{23})^T=\lambda (b_3,a_{31},a_{32},a_{33})^T$ for some $\lambda\in\mathbb{R}$;
\item[4.] $H=A_{33}x_1x_3 + \beta e_1 \ln|x_3|-\gamma e_1 \ln|x_2|$, where $\beta,\gamma$ are solutions of 
$ B_3=0$, $A_{31}=0$, $A_{32}=0$, under conditions\\
$A_{33}\neq0$, $b_1+b_3=a_{1i}+a_{3i}=0$ for $i=1,2,3$ and $(b_1,a_{11},a_{12})^T=\lambda (b_2,a_{21},a_{22})$ for some $\lambda\in\mathbb{R}$;
\item[5.] $H=(a_{13}+a_{23})x_3-(a_{12}+a_{32})x_2 + e_1 \ln|x_3| - e_1 \ln|x_2|$, under conditions\\
$a_{12}+a_{32}\neq0$, $a_{13}+a_{23}\neq 0$, $b_1+b_2=b_1+b_3=0$,
$a_{11}+a_{21} = a_{11}+a_{31}=0$ and $a_{12}+a_{22}=a_{13}+a_{33}=0$;
\item[6.] $H=-a_{23}x_1x_2^{l_2}x_3 + e_1 x_2^{l_2} / l_2$, where $l_2=-(a_{13}+a_{33})/a_{23}$, under conditions\\
$a_{23}\neq0$, $a_{13}+a_{33}\neq0$, $b_1=a_{21}=a_{22}=0$ and $b_1+b_3=a_{11}+a_{31}=a_{12}+a_{32}=0$;
\item[7.] $H=(b_1+a_{11}x_1)x_1x_2^{l_2}x_3^{l_3} + e_1 x_2^{l_2}x_3^{l_3}$, where $l_2,l_3$ are solutions of
$b_2 l_2 + b_3 l_3 = -b_1$, $a_{21} l_2 + a_{31} l_3 = -2a_{11}$, 
$a_{22} l_2 + a_{32} l_3 = a_{23} l_2 + a_{33} l_3=0$, under conditions\\
$a_{12}=a_{13}=0$, $a_{22}(-b_1a_{31}+2b_3a_{11})+a_{32}(-2b_2a_{11}+b_1a_{21})=0$
and $a_{23}(-b_1a_{31}+2b_3a_{11})+a_{33}(-2b_2a_{11}+b_1a_{21})=0$;
\item[8.] $H=(a_{13}+a_{33})x_1x_2^{l_2}x_3^{l_3+1} + e_1 x_2^{l_2}x_3^{l_3} $, 
where $l_2,l_3$ are given below 
\[
    l_2 = \frac{\gamma(a_{13}+a_{33})}{A_{23}} \quad l_3 = -\frac{\beta(a_{13}+a_{33})}{A_{23}},
\]
and $\beta,\gamma$ are solutions of equations $B_3=0$, $A_{31}=0$, $A_{32}=0$ under the conditions\\
$A_{33}\neq0$, $A_{23}\neq0$, $a_{13}+a_{33}\neq0$, $b_1+b_3=a_{11}+a_{31}=a_{12}+a_{32}=0$ and $(b_1,a_{11},a_{12})^T=\lambda
(b_2,a_{21},a_{22})^T$ for some $\lambda\in\mathbb{R}$;
\item[9.] $H=((a_{12}+a_{22})x_2/l_3 + ( a_{13}+a_{23})x_3/(l_3+1)  )x_1x_2^{l_2}x_3^{l_3} + e_1 x_2^{l_2}x_3^{l_3}/l_3$, where $l_2=-(a_{12}+a_{22})/(a_{22}-a_{32})$ and $l_3=-l_2$, under conditions\\
$a_{12}+a_{22}\neq0$, $a_{22}-a_{32}\neq0$, $a_{13}+a_{33}\neq0$, $a_{23}-a_{33}\neq0$,
$b_1+b_3=b_2-b_3=0$, $a_{11}+a_{31}=a_{21}-a_{31}=0$
and $(a_{13}+a_{33})(a_{22}-a_{32})-(a_{12}+a_{22})(a_{23}-a_{33})=0$.
\end{enumerate}
\end{lemma}
We note that the integrals in the lemma \ref{lemma100} point 3 and 5 do not depend on the variable $x_1$. They immediately follow from the two-dimensional case. 

\subsection{The case $e_1=e_2=e_3=0$}
Finally, we shall discuss the case where all the constant terms are zero. The equations (\ref{cond6m} - \ref{cond6o}) 
are satisfied immediately. This means we only need to find conditions on the parameters in order to satisfy equations
(\ref{cond6a} - \ref{cond6l}). 

In the following lemma, we describe our results obtained using an integrating factor matrix of the form $T_2$.
\begin{lemma}\label{lemma000}
The vector field (\ref{eq:lv_e3}) with $e_1=e_2=e_3=0$ has the following first integrals:
\begin{enumerate}
\item[1.] $H=\alpha(b_1\ln|x_2|-b_2\ln|x_1|-a_{21}x_1)+\beta(b_1\ln|x_3|-b_3\ln|x_1|-a_{31}x_1)$ where $\alpha,\beta$ are 
          solutions of equations $A_{22}=0$ and $A_{23}=0$, under conditions\\
	  $b_1\neq0$, $a_{11}=a_{12}=a_{13}=0$ and $a_{22}a_{33}-a_{32}a_{23}=0$;
\item[2.] $H=\alpha(b_2/x_1-a_{21}\ln|x_1|+a_{11}\ln|x_2|)+\beta(b_3/x_1-a_{31}\ln|x_1|+a_{11}\ln|x_3|)$ where 
          $\alpha,\beta$ are solutions of equations $A_{22}=0$, $A_{23}=0$, under conditions\\
	  $a_{11}\neq0$, $b_1=a_{12}=a_{13}=0$ and $a_{22}a_{33}-a_{32}a_{23}=0$;
\item[3.] $H=A_{31}\ln|x_3|+A_{11}\ln|x_2|-A_{21}\ln|x_1|+A_{22}x_2/x_1$ where $\alpha,\beta,\gamma$ are solutions of 
          $B_1=0$, $A_{13}=0$ and $\beta+\gamma=0$, under conditions\\
	  $A_{11}^2+A_{31}^2\neq0$, $A_{22}\neq0$, $b_1-b_2=a_{12}-a_{22}=a_{13}-a_{23}=0$ and $b_1a_{33}-b_3a_{13}=0$;
\item[4.] $H=-(a_{11}-a_{21})\ln|x_3|+(a_{11}-a_{31})\ln|x_2|-(a_{21}-a_{31})\ln|x_1|+(a_{22}-a_{32})x_2/x_1-(a_{13}-a_{23})x_3/x_1$, 
	  under conditions\\
	  $(a_{11}-a_{31})^2+(a_{11}-a_{21})^2\neq0$, $a_{22}-a_{32}\neq0$, $a_{23}-a_{33}\neq0$ and $b_1-b_3=b_1-b_2=a_{12}-a_{22}=a_{13}-a_{33}=0$;
\item[5.] $H=x_1^{\beta}x_2^{\gamma}$, where $\beta,\gamma$ are solutions of $B_3=0$, 
          $A_{31}=0$, $A_{32}=0$ and  $A_{33}=0$, under conditions\\
	  $(b_1,a_{11},a_{12},a_{13})^T=\lambda(b_2,a_{21},a_{22},a_{23})^T$ for some $\lambda \in \mathbb{R}$;
\item[6.] $H=x_1^{\beta}x_2^{\gamma+\alpha}x_3^{\beta}$, where $\alpha,\beta,\gamma$ are solutions of 
          $B_2=0$, $B_3=0$, $A_{2i}=0$ and $A_{3i}=0$ ($i=1,2,3$), under conditions\\
	  $(b_1,a_{11},a_{12},a_{13})^T=\lambda_1(b_2,a_{21},a_{22},a_{23})^T$ and\\
          $(b_2,a_{21},a_{22},a_{23})^T=\lambda_2(b_3,a_{31},a_{32},a_{33})^T$ for some $\lambda_1, \lambda_2 \in \mathbb{R}$;
\item[7a.] $H=x_1^{l_1}x_2^{l_2}(A_{12}x_2/(l_2+1)+A_{33}x_3)$ where $l_1=-\beta(-a_{23}+a_{33})/A_{33}$,  
           $l_2=-A_{13}/A_{33}$ and $\alpha,\beta,\gamma$ are solutions of $B_3=0$, $A_{31}=0$, $A_{32}=0$ and $\alpha+\beta=0$, under conditions\\
	   $A_{33}\neq0$, $A_{13}\neq0$, $A_{23}\neq0$, $A_{12}\neq0$, $b_2-b_3=a_{21}-a_{31}=0$ and 
           $(b_1,a_{11},a_{12})^T=\lambda(b_2,a_{21},a_{22})^T$ for some $\lambda \in \mathbb{R}$;
\item[7b.] $H=x_1^{l_1}x_2^{l_2}(-b_3\beta/l_1+A_{33}x_3)$ where $l_1=-a_{33}\beta/A_{33}$, $l_2=-(a_{33}\gamma)/A_{33}$ 
           and $\beta,\gamma$ are solutions of $B_3=0$, $A_{31}=0$, $A_{32}=0$, under conditions\\
	   $b_3\neq0$, $a_{33}\neq0$, $A_{33}\neq0$, $a_{31}=a_{32}=0$, and\\ 
           $(b_1,a_{11},a_{12})^T=\lambda(b_2,a_{21},a_{22})^T$ for some $\lambda \in \mathbb{R}$;
\item[7c.] $H=x_1^{l_1}x_2^{l_2}((a_{21}-a_{31})x_1/(l_1+1)+ (a_{22}-a_{32})x_2/l_2 +(a_{13}-a_{23})x_3)$ 
           where $l_1=(a_{23}-a_{33})/(a_{13}-a_{23})$, $l_2=(a_{33}-a_{13})/(a_{13}-a_{23})$, under conditions\\
	   $a_{11}-a_{31}\neq0$, $a_{12}-a_{32}\neq0$, $a_{13}-a_{33}\neq0$, $a_{23}-a_{33}\neq0$, $a_{13}-a_{23}\neq0$, $b_1-b_2=b_1-b_3=0$ 
	   and $a_{11}-a_{21}=a_{12}-a_{22}=0$;
\item[7d.] $H=x_1^{l_1}x_2^{l_2}(-(b_2+b_3)/(l_1)-(a_{21}+a_{31})x_1/l_2+(a_{13}+a_{23})x_3)$, where 
           $l_1 = -(a_{23}+a_{33})/(a_{13}+a_{23})$ and $l_2 = -(a_{33}-a_{13})(a_{13}+a_{23})$, under conditions\\
	   $a_{13}+a_{23} \neq 0$, $b_1+b_2=0$, $a_{11}+a_{21}=0$, $a_{12}+a_{22}=0$, $a_{12}-a_{32}=0$ and $(b_1-b_3)a_{23}-(b_2+b_3)a_{13}=0$;
\item[8a.] $H=x_1^{l_1}x_2^{l_2}x_3^{l_3}(b_1 + a_{11}x_1)$ where $l_1=-(a_{11}B_2)/(a_{11}B_2-b_1A_{21})$, 
           $l_2 = -(b_1\alpha)/(B_2)$, $l_3=(b_1\beta)/(B_2)$, and $\alpha,\beta$ are solutions of 
           $A_{22}=0$, $A_{23}=0$, under conditions\\
	   $b_1^2+a_{11}^2\neq0$, $B_2\neq0$, $a_{11}B_2-b_1A_{21}\neq0$, $a_{12}=a_{13}=0$ and $a_{22}a_{33}-a_{32}a_{23}=0$;
\item[8b.] $H=x_1^{l_1}x_2^{l_2}x_3^{l_3} (a_{12}x_2/l_3-a_{13}x_3/l_2)$ where 
\[
    l_1=\frac{(a_{22}-a_{32})(a_{23}-a_{33})}{-a_{12}(a_{23}-a_{33})+a_{13}(a_{22}-a_{32})} \quad
    l_2=\frac{a_{12}(a_{23}-a_{33})}{-a_{12}(a_{23}-a_{33})+a_{13}(a_{22}-a_{32})},
\]
and $l_2+l_3=-1$, under conditions\\
$a_{13}\neq0$, $a_{23}\neq0$, $a_{22}-a_{32}\neq0$, $a_{23}-a_{33}\neq0$, $b_1=a_{11}=0$ and $b_2-b_3=a_{21}-a_{31}=0$;
\item[8c.] $H=x_1^{l_1}x_2^{l_2}x_3^{l_3}((b_1+b_2)/(l_3)+(a_{11}+a_{21})x_1/l_3+(a_{12}+a_{22})x_2/l_3-(a_{23}+a_{33})x_3/l_1)$
where 
\[
    l_1=\frac{A_{22}(A_{21}-A_{11})}{A_{11}A_{22}-A_{21}A_{12}} \quad 
    l_2=\frac{A_{11}(A_{12}-A_{22})}{A_{11}A_{22}-A_{21}A_{12}} \quad 
    l_3=\frac{A_{31}(A_{33}-A_{23})}{A_{31}A_{23}-A_{21}A_{33}},
\]
where $A_{11}A_{22}-A_{21}A_{12}\neq0$, $A_{31}A_{23}-A_{21}A_{33}\neq0$ and  $\alpha=\beta=\gamma$  under conditions\\
$B_1A_{22}(A_{21}-A_{11})+B_2A_{11}(A_{12}-A_{22})=0$, $A_{13}A_{22}(A_{21}-A_{11})$, $+A_{23}A_{11}(A_{12}-A_{22})=0$, $B_3A_{23}(A_{21}-A_{31})+B_2A_{31}(A_{33}-A_{23})=0$ and $A_{32}A_{23}(A_{21}-A_{31})+A_{22}A_{31}(A_{33}-A_{23})=0$ with $\alpha=\beta=\gamma$.
\end{enumerate}
\end{lemma}
Remark that the integral in the Lemma \ref{lemma000} no 5. is the same integral we found in the two-dimensional LV case. 

\subsection{On the first integrals of three-dimensional LV systems}
We here give some remarks about lemma \ref{lemma000}, in which we discussed first integrals of three-dimensional LV systems that have already been extensively discussed in the literature~\cite{ollagnier,cairo92families,cairo00,plank,gao00,gao98,gao99,Gonzalez-Gascon00,grammaticos,cairo3d,cairo_nonlinearmathphy00}. The forms of first integrals obtained in lemma \ref{lemma000} (points 1-4) are similar to the ones obtained by Plank (1995) \cite{plank}, which in fact, are special cases \cite{cairo_comment} of invariants found in Cair{\'o} and Feix 
\cite{cairo92families}. It is not hard to show that their integrals are different from the ones we have obtained in this paper.

The other integrals in the lemma \ref{lemma000} (points 5-8) have the following general form 
\begin{equation}\label{form:Dar}
x_1^{\lambda_1}x_2^{\lambda_2}x_3^{\lambda_3}\,\varphi(x_1,x_2,x_3)^{\lambda_4} 
\end{equation} 
where $\lambda_1$, $\lambda_2$, $\lambda_3$ and $\lambda_4$ are some constants and $\varphi$ is a polynomial function of degree one in $x_1,x_2,x_3$. To our best knowledge, the above general form first appeared as a first integral in Cair{\'o} and Feix \cite{cairo92families}, except the fact that the first integral in their paper is time-dependent. Through a time-rescaling (see Hua et al. \cite{hua93}) they obtained a time-independent first integral. The existence of first integrals of this form of three-dimensional LV systems is also extensively investigated by Cair{\'o} \cite{cairo_nonlinearmathphy00}. He investigated a polynomial function $\varphi$ of degree one and two. The integral functions in point 5-6 has the same form with the ones he found \cite[Theorem 2(1)]{cairo_nonlinearmathphy00}. The forms of integrals in point 7 do not seem to have been recognized before, thus these new results extend the known results on integrals of the form (\ref{form:Dar}). The integral of the form 8a generalizes integrals obtained in \cite[Theorem 2(8-13)]{cairo_nonlinearmathphy00}. Finally, integrals of the form 8b and 8c seem to be new.

\section{Conclusion}
In this paper, we have derived first integrals for two and three-dimensional LV systems with constant terms 
through an integrating factor matrix. We make Ansatzs and obtain conditions for the existence of first integrals. 
By our method, the search for integrals in dynamical systems changes into a linear algebra problem.
We note that some conditions of the two and three LV system with constant terms to have first integrals 
do not involve any constant terms. This is good because under such conditions tfirst integrals are preserved 
for any constant terms.

In the two-dimensional case, the integrating factor matrix $T(x_1,x_2)$ in equation (\ref{eq:intmatrix}), along with the condition such that $T$ is an integrating factor, turns out to be similar to the one, used by Plank~\cite{plank}  due to $S=T^{-1}$. In the three-dimensional case this property no longer applies, as a $3 \times 3$ skew-symmetric matrix is not invertible. Hence, our analysis generalises the work of Plank's.

\noindent {\bf Comparison with Darboux method} 
The Darboux method has been applied to find first integrals of two-dimensional \cite{cairo99,cairo97} and three-dimensional \cite{cairo3d,cairo_nonlinearmathphy00} LV systems. Compared to this method, our method has advantages in the context of searching for a first integral of LV systems with constant terms. To apply the Darboux method, one must seek an algebraic curve of a vector field, and LV systems without constant terms have natural algebraic curves as the axes are invariant. For systems with constant terms, this no longer applies.

\noindent {\bf Comparison with Hamiltonian method} 
The existence of first integrals of two-dimensional systems has been obtained using Hamiltonian method \cite{hua93,cairo92ontheHam,cairo93}. They assumed that the integrals are products of two or three polynomial functions of degree one. 
The advantage of our method is that we only make a single Ansatz.
Gao \cite{gao99} has also applied a Hamiltonian method to find first and second integrals of special cases of three-dimensional LV systems, where the linear terms are absent. 

\noindent {\bf Comparison with Frobenius method} The Frobenius method was first introduced by Strelcyn and Wojciechowski~\cite{strelcyn88} to find a first integral for three-dimensional systems. It has been used to find integrals for LV systems by Grammaticos et al.~\cite{grammaticos}. Unfortunately, they only look at a special case of the LV system, which is the so-called ABC system. 

\noindent {\bf Comparison with Carleman embedding method} The existence of first integrals in $n$-dimensions has been studied \cite{cairo92families,cairo89} through the Carleman embedding method. However, the integrals that are obtained are time-dependent and are refered to as \emph{invariants}. 

For future work, it would be a challenging problem to find a more general integrating factor matrix of a vector field in dimension greater than or equal to two. 
Several first integrals with the corresponding conditions have been published in the literature and we could use these results to endeavour to find the general integrating factor matrix. 

\ack{These investigations were funded by the Australian Research Council. Author LvV was also supported by the National Science and Engineering Research Council, grant nr. 355849-2008}

\section*{References}
\bibliographystyle{unsrt}
\bibliography{first_integral}
\end{document}